\documentclass[12pt,a4paper]{amsart}
\usepackage[english]{babel}
\usepackage[latin1]{inputenc}
\usepackage{amsfonts}
\usepackage{amssymb}
\usepackage{amsmath}
\usepackage{amsthm}
\usepackage{enumerate}
\usepackage{float}
\usepackage{hyperref}

\newtheorem{theorem}{Theorem}[section]

\newtheorem{cor}[theorem]{Corollary}

\newtheorem{rem}[theorem]{Remark}
\newtheorem{prop}[theorem]{Proposition}
\newtheorem{defi}[theorem]{Definition}

\usepackage{tikz}
\usetikzlibrary{shapes,shadows,arrows,decorations.markings,calc}

\begin{document}
\title{On approximation schemes and compactness}
\author{Asuman G. Aksoy and Jose M. Almira}
\date{}

\maketitle

\begin{abstract}
We present an overview of some results about characterization of compactness in which the concept of approximation scheme has had a role. In particular,  we present several results that were proved by the second author, jointly with Luther, a decade ago, when these authors were working on a very general theory of approximation spaces. We then introduce and show  the basic properties of a new concept of compactness, which was studied by the first author in the eighties, by using a generalized concept of approximation scheme and its associated Kolmogorov numbers, which generalizes the classical concept of compactness.
\end{abstract}

\section{Motivation}
One of the basic notions in functional analysis is compactness. Its utility  has become of fundamental importance after the appearance of Arzel\`{a}-Ascoli's Theorem \cite{arzela}, \cite{ascoli} especially pointing its use for the proof of existence results when investigating the solutions of differential equations. Indeed, a key step for the proof of convergence in many algorithms is precisely to show that a certain set is compact, and many theorems have been produced to characterize compactness of subsets of the numerous function spaces and operator spaces that appear in functional analysis. The compactness of operators was also a main ingredient for the study of the solutions of integral equations, and was indeed introduced by Hilbert in his studies of the equations of Mathematical Physics.  In particular, Hilbert and his student Schmidth proved a very nice decomposition formula for all self-adjoint compact operator $T:H\to H$, where $H$ is any separable Hilbert space: the spectral decomposition theorem. This theory was soon investigated and amplified to a beautiful set of results which we call nowadays Riesz theory (or Riesz-Schauder Theory) and is devoted to the study of operators $S:X\to X$ (where $X$ denotes any complex Banach space) that can be expressed as $S=\lambda I_{X}-T$ with $\lambda\neq 0$ (an scalar) and $T:X\to X$, a compact operator. In such study, the spectral properties of the operator $T$ are essential and, in connection with these properties, it was soon discovered that some entropy and approximation quantities were of great importance (see, e.g., \cite{Carl_Stephani} for a detailed study of this connection). Compactness has also been a fundamental concept for the development of other parts of Mathematical Analysis, such as Fixed Point Theory or Approximation Theory. Concretely, Brouwer's fixed point theorem \cite{Br} asserts that every compact convex set $K$ in $\mathbb{R}^n$ is a fixed point space, that is, if $f:K\to K$ is continuous, then $f(x)=x$ for some $x\in K$ (see \cite[p. 25]{Mat} for a nice easy demonstration). On the other hand, Schauder's fixed point theorem \cite{Sh}, which has numerous
applications in Mathematical Analysis, asserts that every convex set in a normed linear space is a fixed point space for compact maps (see also \cite{Bon}). Among the results equivalent to Brouwer's fixed point theorem, the theorem of Knaster, Kuratowski and Mazurkiewicz (in short, KKM) \cite{KKM} occupies a special place. Ky Fan, using KKM maps, was able to prove a best approximation theorem \cite{KF}. Later on, this concept was generalized by  Khamsi to metric space setting by demonstrating a result which can be seen as an extension of Brouwer and Schauder's fixed point theorems (see \cite{Kh}). Finally, just to include  in this section some results related to Approximation Theory, we would like to stand up that compactness of natural embeddings $Y \hookrightarrow X$ is, in fact, the main reason because, in many classical contexts, we can prove that approximation errors (with respect to arbitrary approximation schemes) and Fourier coefficients of functions that belong to the space $Y$, decay to zero with a certain prescribed behavior. This was recently proved by Almira and Oikhberg \cite{almira_oikhberg_2} and by Almira \cite{almira_fourier_coeff}.

In this paper, we survey some results about the characterization of compactness in which the concept of approximation scheme has had a role. Concretely, in Section 2 we present several results that were proved by the second author, jointly with Luther, a decade ago, when these authors were working on a very general theory of approximation spaces \cite{almiraluther2}, \cite{almiraluther1} (see also \cite{fugarolas}) and, in Section 3, we introduce and show  the basic properties of a new concept of compactness, which was studied by the first author in the eighties  \cite{Ak0}, \cite{Ak1}, \cite{Ak2}, \cite{AkNa}, by using a generalized concept of approximation scheme and its associated Kolmogorov numbers, which generalizes the classical concept of compactness.

\section{Approximation schemes, approximation spaces and compactness}
\subsection{Preliminaries}
\begin{defi} \label{ap_sch} Given $(X,\|\cdot\|)$  a quasi-Banach space, and
$A_0\subset A_1\subset\ldots \subset A_n\subset\ldots \subset X$
an infinite chain of subsets of $X$, where all inclusions are
strict,  we say that $(X,\{A_n\})$ is an {\it approximation scheme} (or that $(A_n)$ is an approximation scheme in $X$) if:
\begin{itemize}
\item[$(A1)$] there exists a map $K:\mathbb{N}\to\mathbb{N}$ such that $K(n)\geq n$ and $A_n+A_n\subseteq A_{K(n)}$ for all $n\in\mathbb{N}$,

\item[$(A2)$] $\lambda A_n\subset A_n$ for all $n\in\mathbb{N}$ and all scalars $\lambda$,

\item[$(A3)$] $\bigcup_{n\in\mathbb{N}}A_n$ is a dense subset of $X$.
\end{itemize}
We say that the approximation scheme $(X,\{A_n\})$ is nontrivial if $A_n=\overline{A_n}^X\subsetneq A_{n+1}$ for all $n$, and we say that it is linear if $A_n$ is a vector subspace of $X$ for all $n$.
\end{defi}
Approximation schemes were introduced in Banach space theory by Butzer and Scherer in 1968 \cite{butzer_scherer} and, independently,  by Y. Brudnyi and N. Kruglyak under the name of ``approximation families'' in 1978 \cite{brukru}. They were
popularized by Pietsch in his 1981 seminal paper \cite{Pie}, in which he introduced the approximation spaces
$$A_p^r(X,{A_n})=\{x\in X: \|x\|_{A_p^r}=\|\{E(x,A_n)\}_{n=0}^\infty\|_{\ell_{p,r}}<\infty\},$$ where $$\ell_{p,r}=\{\{a_n\}\in\ell^\infty: \|\{a_n\}\|_{p,r}=\left[\sum_{n=1}^\infty n^{rp-1}(a_n^*)^p\right]^\frac{1}{p}<\infty\}$$ denotes the so-called Lorentz sequence space (in particular, $\{a_n^*\}$ is the non-increasing rearrangement of $\{a_n\}$), $(X,\|\cdot\|_X)$ is a quasi-Banach space, and $E(x,A_n)=\inf_{a\in A_n}\|x-a\|_X$.

There were two main motivations for Pietsch's study of approximation spaces. On the one hand, the spaces $A_p^r(X,{A_n})$ form a scale which allows a natural interpretation of the so called central theorems in approximation theory as the appropriate tool for the classification of functions and operators in terms of their smoothness (compactness, respectively) properties, which crystallize with the property of  membership to one of these spaces (see, for example,  \cite{almirabjma}, \cite{devorenonlinear}, \cite{devore}, \cite{PieID}). On the other hand, he also detected a very nice  parallelism between the theories of approximation spaces and interpolation spaces. In particular, he proved embedding, reiteration and representation results for his approximation spaces.

Simultaneously and also independently, Ti\c{t}a \cite{tita_col2} studied, from 1971 on, for the case of approximation of linear operators by finite rank operators,  a similar concept, based on the use of symmetric norming functions $\Phi$ %(see \cite{salinas} for the definition)
and the sequence spaces defined by them, $S_{\Phi}=\{\{a_n\}:\exists \lim_{n\to\infty}\Phi(a_1^*,a_2^*,\cdots,a_n^*,0,0,\cdots)\}$ and, later on, Almira and Luther \cite{almiraluther2}, \cite{almiraluther1} developed a  theory for generalized approximation spaces via the use of general sequence spaces $S$ (that they named ``admissible sequence spaces'') and  defined  approximation spaces as $$A(X,S,\{A_n\})=\{x\in X: \|x\|_{A(X,S)}=\|\{E(x,A_n)\}\|_S<\infty\}.$$ Furthermore, this theory, which also includes the reiteration and representation theorems, was developed by the authors without using any result from interpolation theory. Admissibility of the sequence space $S$ is just a technical imposition that allows to prove that $\|x\|_{A(X,S)}=\|\{E(x,A_n)\}\|_S$ defines a quasi-norm. This property is automatically satisfied by the sequence spaces $S$ which contain all finite null sequences and satisfy that, if
$\{b_n\}\in S$ and $|a_n|\leq |b_n|$ for all $n$, then $\{a_n\}\in S$ and $\|\{a_n\}\|_{S}\leq \|\{b_n\}\|_S$; if $K(n)=n$ (see \cite[Definition 3.2]{almiraluther2}). Other papers with a similar spirit of generality have been written by Aksoy \cite{Ak0}, \cite{Ak1}, \cite{Ak2}, \cite{AkNa}, Ti\c{t}a \cite{tita_cluj_99} and Pustylnik \cite{pustylnik}, \cite{pustylnik2}. Finally, a few other important references for people interested on approximation spaces and/or approximation schemes are \cite{almira_oikhberg},  \cite{almira_oikhberg_2}, \cite{cobos}, \cite{cobos_milman}, \cite{cobos_resina}, \cite{feher_g}, \cite{Oik}, \cite{peetre_sparr}, \cite{tita_anal}, \cite{tita_col}, \cite{tita_col2} and \cite{tita_studia}.
It is important to remark that, due to the centrality of the concept of approximation scheme in approximation theory,  the idea of defining approximation spaces is a quite natural one. Unfortunately, this has had the negative effect that many unrelated people has thought on the same things at different places and different times, and some papers on this subject partially overlap.

Along this paper we will assume that all spaces appearing are normed, although many of the results presented here also hold true in the  quasi-normed setting.

\subsection{Characterization of compactness with boundedly compact approximation schemes and the Arzel\`{a}-Ascoli Theorem}

A first characterization of compactness in complete metric spaces was given by Hausdorff, who proved that $M$ is relatively compact in the complete metric space $(X,d)$ if and only if for every $\varepsilon>0$ there exists a finite $\varepsilon$-net for $M$  (i.e., a finite set of points $\{x_k\}_{k=1}^s\subseteq X$ such that  $M\subseteq \bigcup_{k=1}^NB_d(x_k,\varepsilon)$, where $B_d(x,t)=\{y\in X:d(x,y)\leq t\}$).  This result can be reformulated as a characterization of compactness with the aid of approximation schemes as follows.

%A natural characterization of compactness, which rest on the characterization of finite dimensional Banach spaces as those Banach spaces whose unit ball is compact, is the following one:

\begin{theorem} \label{timan} Assume that $(X,\{A_n\})$ is an approximation scheme with $ A_n $ boundedly compact  for all $n\in\mathbb{N}$, and let $M\subseteq X$. Then the following are equivalent claims:
\begin{itemize}
\item[$(i)$] $M$ is a relatively compact subset of $X$
\item[$(ii)$] $M$ is a bounded subset of $X$ and $\lim_{n\to\infty}E(M,A_n)=0$.
\end{itemize}
Furthermore, the implication  $(i)\Rightarrow (ii)$ holds true for arbitrary approximation schemes $\{A_n\}$.
\end{theorem}

\begin{proof}
$(i)\Rightarrow (ii)$ Assume that $M\subseteq X$ is relatively compact. Then $M$ is bounded in $X$ since $\overline{M}^X$ is bounded (compactness implies boundedness). We must show that $\lim_{n\to\infty}E(M,A_n)=0$. Take $\varepsilon >0$ and let $\{x_1,\cdots,x_N\}\subseteq X$ be an $\varepsilon$-net for $M$. Then, given $x\in M$, $E(x,A_n)\leq E(x-x_k,A_n)+E(x_k,A_n)\leq \varepsilon +\max_{k=1,\cdots,N}E(x_k,A_n) \leq 2\varepsilon$ for $n\geq N_0(\varepsilon)$,  since
$\lim_{n\to\infty}\max_{k=1,\cdots,N}E(x_k,A_n)=0$. Note that we have used nothing about $\{A_n\}$ but the fact that $\bigcup_{n\in\mathbb{N}}A_n$ is a dense subset of $X$.

$(ii)\Rightarrow (i)$ Let $\varepsilon>0$ be an arbitrary positive constant. By hypothesis, there exists $N_0=N_0(\varepsilon)>0$ such that $E(M,A_{N_0})<\varepsilon/4$. In particular, every $x\in M$ admits a decomposition $x=a(x)+y(x)$ with $a(x)\in A_{N_0}$ and $\|y(x)\|=\|x-a(x)\|\leq \varepsilon/2$.  Now, boundedness of $M$ implies that there exists a constant $C>\varepsilon$ such that $M\subseteq CU_X$, so that $\|a(x)\|\leq \|x\|+\|y(x)\|\leq C+\varepsilon/2\leq 2C$.

Let
$\{b_1,b_2,\cdots,b_s\}$ be a $\varepsilon/2$-net in $A_{N_0}\cap 2CU_X$, which is a compact set since $A_{N_0}$ is boundedly compact.  Given $x\in M$ there exists $i\leq s$ such that $\|a(x)-b_i\|\leq \varepsilon/2$, so that
\[
\|x-b_i\|\leq \|x-a(x)\|+\|a(x)-b_i\|\leq \varepsilon,
\]
which proves that $\{b_1,\cdots,b_s\}$ is a finite $\varepsilon$-net for $M$. Hausdorff's theorem guarantees that $M$ is relatively compact in $X$.
\end{proof}

\begin{cor} \label{A_nCompatness} Assume that $(X,\{A_n\})$ is an approximation scheme with $ A_n $ boundedly compact  for all $n\in\mathbb{N}$.  A set $M\subseteq X$ satisfies $\lim_{n\to\infty}E(M,A_n)=0$ if and only if there exists $M'$, a relatively compact  subset of $X$, and a natural number $N\in\mathbb{N}$  such that $M\subseteq A_N+M'$ is satisfied.
\end{cor}

\begin{proof} Assume that $M\subseteq A_N+M'$ with $M'$ relatively compact  in $X$. Then Theorem \ref{timan} implies that $\{E(M',A_n)\}\searrow 0$. Take $x\in M$ and $n\in\mathbb{N}$, $n\geq N$. Then there exists $a\in A_N$, $y\in M'$ such that $x=a+y$ and
\begin{eqnarray*}
E(x,A_{K(n)}) &=& E(a+y,A_{K(n)}) \\
&\leq& E(a,A_n)+E(y,A_n) \\
&=& E(y,A_n)\leq E(M',A_n),
\end{eqnarray*}
so that $E(M,A_{K(n)})\leq E(M',A_n)$ for all $n\geq N$, and $\{E(M,A_n)\}\searrow 0$.

Let us now assume that  $\{E(M,A_n)\}\searrow 0$. If $M$ is a bounded subset of $X$ then Theorem \ref{timan} implies that $M$ is relatively compact, so that we can take $M'=M$ and $N=0$. On the other hand, if $M$ is unbounded, then we can take $N\in\mathbb{N}$ such that $E(M,A_N)\leq 1/2$ and define $M'=\{y\in U_X: \text{ exists } x\in M \text{ and } a\in A_N \text{ such that } y=x-a\}$. $M'$ is obviously bounded and, if $y=x-a\in M'$ with $a\in A_N$, $x\in M$, then, for each $n\geq N$,
\begin{eqnarray*}
E(y,A_{K(n)}) &=& E(x-a,A_{K(n)}) \\
&\leq& E(a,A_n)+E(x,A_n) \\
&=& E(x,A_n)\leq E(M,A_n),
\end{eqnarray*}
which proves that $E(M',A_{K(n)})\leq E(M,A_n)$ for all $n\geq N$. Thus $\{E(M',A_n)\}\searrow 0$ and Theorem \ref{timan} implies that $M'$ is a relatively compact subset of $X$.
\end{proof}

\begin{cor}[Arzel\`{a}-Ascoli]  A set $M\subseteq C[a,b]$ is relatively compact in $C[a,b]$ if and only if it is uniformly bounded and equicontinuous.
\end{cor}
%Let $K$ be a compact metric space and let $C(K)$ denote the space of continuous functions $f:K\to \mathbb{R}$ with the uniform norm.
%$C(K)$ is separable, so that it admits a linear approximation scheme $\{A_n\}$ with $\dim A_n<\infty$ for all $n$.

\begin{proof}
Let us consider the approximation scheme  $(C[a,b],\{\Pi_n\})$, where $\Pi_n$ denotes the space of (algebraic) polynomials of degree $\leq n$ and let us assume that $M$ is relatively compact in $C[a,b]$. Then, Theorem \ref{comp_dim_fin} implies that $M$  is a bounded subset of $C[a,b]$ (i.e., $M$ is uniformly bounded, so that there exists $C>0$ such that $\|f\|_{C[a,b]}\leq C$ for all $f\in M$) and $\{E(M,\Pi_n)\}\searrow 0$. Let us show that $M$ is equicontinuous.

Given $\varepsilon>0$ (without loss of generality we assume $\varepsilon<C$), there exists $N\in\mathbb{N}$ such that $E(M,\Pi_N)<\varepsilon/8$.  Furthermore, for all $t,s\in [a,b]$ and all $f\in M$ we have
\[
|f(t)-f(s)|\leq |f(t)-p(t)|+|p(t)-p(s)|+|p(s)-f(s)| \text{ for all } p\in \Pi_N.
\]
Hence, if we take $p=p^*\in \Pi_N$ such that $\|f-p^*\|_{C[a,b]}\leq 2E(f,\Pi_N)$, then
\[
|f(t)-f(s)|\leq 4E(f,\Pi_N)|+|p^*(t)-p^*(s)|\leq \varepsilon/2+w(p^*,|t-s|),
\]
where $w(h,\delta)=\sup_{|t-s|\leq \delta}|h(t)-h(s)|$ denotes the modulus of continuity of the function $h$. Now, if $p(t)=a_0+a_1t+\cdots+a_Nt^N\in\Pi_N$, then both $\|p\|_{0}=\max_{0\leq k\leq N} |a_k| $ and $\|p\|_1=\|p\|_{C[a,b]}$ define a norm over the finite dimensional space $\Pi_N$, so that they are equivalent norms.  On the other hand, $M$ being bounded, the norm of $p^*$ must be controlled by a constant $K>0$ (since $\|p^*\|\leq \|p^*-f\|+\|f\|\leq \varepsilon/4+C\leq 2C=K$). This implies that we can assume $\max_{0\leq k\leq N}|a_k|\leq K^*$ for a certain constant $K^*>0$ and hence
\[
w(p^*,|t-s|)\leq K^*\sum_{k=0}^{N}w(\phi_k,|t-s|); \text{ where  }\phi_k(x)=x^k, \ k=0,1,\cdots,N.
\]
(since $w(ah_1+bh_2,\delta)\leq
\max\{|a|,|b|\}(w(h_1,\delta)+w(h_2,\delta))$ for all scalars $a,b$ and functions $h_1,h_2$, and $p^*=\sum_{k=0}^N\alpha_k\phi_k$). In particular, we can choose $\delta=\delta(\varepsilon)>0$ such that $|t-s|\leq \delta $ implies $\max_{0\leq k\leq N}w(\phi_k,|t-s|)\leq \frac{\varepsilon}{2K^*(N+1)}$. This shows that $w(f,\delta)\leq \varepsilon$ for all $f\in M$, which is what we wanted to prove. To prove the other implication we can use Theorem \ref{comp_dim_fin} with $A_n=\Pi_n$ and the well known Jackson's inequality for algebraic approximation $E(f,\Pi_n)\leq C w(f,\frac{1}{n+1})$, $n=0,1,\cdots$.
\end{proof}

 In this section of the paper, we will concentrate our attention most of the time on linear approximation schemes defined over Banach spaces $X$, since they are enough for the applications we mention explicitly here. In such a case it is known that all sequence spaces
$\ell^q(\beta)=\{\{a_n\}\subset \mathbb{R}: \|\{a_n\}\|_{\ell^q(\beta)}=(\sum_{n=0}^{\infty}b_n|a_n|^q)^{\frac{1}{q}}<\infty\}$ are admissible, so that, when dealing with these spaces we do not worry about the weights $\beta=\{b_n\}\subset [0,\infty)$. Of course, if the approximation scheme is nonlinear and the space  $\ell^q(\beta)$ is not admissible for this approximation scheme, we still can talk about the set $A(X,\{A_n\},\ell^q(\beta))$ and we will say that $M$ is bounded in $A(X,\{A_n\},\ell^q(\beta))$ whenever $\sup_{f\in M}\|\{E(f,A_n\}\|_{\ell^q(\beta)}<\infty$.

\begin{theorem} \label{comp_dim_fin}
Assume that $(X,\{A_n\})$ is an approximation scheme with $ A_n $ boundedly compact  for all $n\in\mathbb{N}$. If  $q\in [1,\infty]$ and $M\subseteq X$, then the following are equivalent statements:
\begin{itemize}
\item[$(i)$] $M$ is a relatively compact subset of $X$.
\item[$(ii)$] There exists $\beta=\{b_n\}_{n=0}^\infty$ a sequence of nonnegative real numbers such that $\|\beta\|_{\ell^q}=\infty$ and  $M$ is a bounded subset of $A(X,\{A_n\},\ell^{q}(\beta))$.
\end{itemize}
\end{theorem}

\begin{proof}
We first show $(i)\Rightarrow (ii)$. If $M$ is relatively compact, then Theorem \ref{timan} proves that $\alpha_n=E(M,A_n)$ satisfies $\{\alpha_n\}\in c_0$ and $E(x,A_n)\leq \alpha_n$ for all $x\in M$ and all $n\in\mathbb{N}$. Thus, if $q=\infty$, then $\sup_{x\in M}\|\{E(x,M)\}\|_{\ell^{\infty}(\{\frac{1}{\alpha_n}\})}\leq 1$ and $M$ is a bounded subset of $A(X,\{A_n\},\ell^{\infty}(\{\frac{1}{\alpha_n}\}))$.

Let us now assume that $q<\infty$. Take $\{n_k\}$ a sequence of natural numbers such that $\alpha_{n_k}\leq 2^{-k}$, $k=1,2,\cdots$ and consider the sequence
$\beta=\{b_n\}$ defined by $b_{n_k}=1$, $k=1,2,\cdots$, and $b_n=\frac{1}{2^n\alpha_n}$ for $n\in\mathbb{N}\setminus\{n_k\}_{k=1}^{\infty}$. Then $\|\beta\|_{\ell^q}^q\geq \sum_{k=1}^\infty b_{n_k}^q=\infty$  and, for each $x\in M$,
\begin{eqnarray*}
\|x\|_{A(X,\{A_n\},\ell^{q}(\beta))}^q & = & \sum_{k=1}^{\infty}E(x,A_{n_k})^q+  \sum_{n\in \mathbb{N}\setminus\{n_k\}_{k=1}^{\infty}}E(x,A_{n})^q(\frac{1}{2^n\alpha_n})^q \\
&\leq &   \sum_{k=1}^{\infty}2^{-kq}+  \sum_{n\in \mathbb{N}\setminus\{n_k\}_{k=1}^{\infty}}\frac{1}{2^{qn}}\leq 3,
\end{eqnarray*}
so that $M$ is a bounded subset of $A(X,\{A_n\},\ell^{q}(\beta))$.

Let us prove $(ii)\Rightarrow (i)$. Let $\beta=\{b_n\}$ be a sequence of nonnegative real numbers such that $b_0>0$ and $\|\beta\|_{\ell^q}=\infty$. Assume that $M$ is a bounded subset of $A(X,\{A_n\},\ell^{q}(\beta))$. Then, $b_0>0$ implies that $M$ is also bounded in $X$. Furthermore, given $x\in M$, we have that
\begin{eqnarray*}
E(x,A_n)\|\{b_k\}_{k=0}^n\}\|_{\ell^q} &\leq & \|\{b_kE(x,A_k)\}_{k=0}^n\|_{\ell^q} \\
&\leq&  \|\{b_kE(x,A_k)\}_{k=0}^\infty\|_{\ell^q}=\|x\|_{A(X,\{A_n\},\ell^{q}(\beta))}\leq C
\end{eqnarray*}
for a certain constant $C$ and all $n\in \mathbb{N}$. This shows that $\{E(M,A_n)\}\searrow 0$, since  $\|\beta\|_{\ell^q}=\infty$ and the estimation above holds for all $x\in M$. Theorem \ref{timan} implies that $M$ is a compact subset of $X$.
 \end{proof}

\begin{cor}Assume that $\|\beta\|_{\ell^q}=\infty$, where $\beta=\{b_n\}$ is a sequence of nonnegative real numbers and $b_0>0$. If $(X,\{A_n\})$ is a linear approximation scheme with $\dim A_n<\infty$ for all $n$,  the embedding $A(X,\{A_n\},\ell^{q}(\beta))\hookrightarrow X$ is compact.
\end{cor}

\begin{proof}
The linearity of $A_n$ guarantees that  $\ell^{q}(\beta)$ is an admissible sequence space for all $\beta$, so that $A(X,\{A_n\},\ell^q(\beta))$ is a Banach space and $A(X,\{A_n\},\ell^{q}(\beta))\hookrightarrow X$ is an embedding. Now the Corollary is just a restatement of the implication $(ii)\Rightarrow (i)$ in Theorem \ref{comp_dim_fin}.
\end{proof}

\begin{cor}
Assume that $(Y,\{A_n\})$ is a linear approximation scheme with $\dim A_n<\infty $  for all $n\in\mathbb{N}$, $X$ is a Banach space, $q\in [1,\infty]$ and $T\in L(X,Y)$. Then the following are equivalent statements:
\begin{itemize}
 \item[$(i)$] $T\in \mathcal{K}(X,Y)$ (i.e., $T$ is a compact operator).
 \item[$(ii)$] There exists a sequence of non-negative real numbers $\beta=\{b_n\}_{n=0}^{\infty}$ such that $b_0>0$, $\|\beta\|_{\ell^q}=\infty$, and $T\in L(X, A(Y,\{A_n\},\ell^q(\beta)))$.
\end{itemize}
\end{cor}

\begin{proof} $(i)\Rightarrow (ii)$. By hypothesis, $T(U_X)$ is relatively compact in $Y$, so that there exists $\beta=\{b_n\}$ is a sequence of nonnegative real numbers such that $\|\beta\|_{\ell^q}=\infty$, $b_0>0$, and $T(U_X)$ is a bounded subset of  $A(Y,\{A_n\},\ell^q(\beta))$. Hence $T\in L(X, A(Y,\{A_n\},\ell^q(\beta)))$.

 $(ii)\Rightarrow (i)$. This implication follows directly from the compactness of the embedding  $A(Y,\{A_n\},\ell^{q}(\beta))\hookrightarrow Y$.
\end{proof}

Theorem \ref{comp_dim_fin}, in conjunction with the reiteration property of approximation spaces, was used by Almira and Luther to prove a compactness criterium for subsets of generalized approximation spaces and, as a corollary, a characterization of convergence in these spaces.

To state these results it is necessary to introduce a little bit more notation. Concretely, given $\beta=\{b_n\}_{n=0}^\infty$ a sequence of positive real numbers, we define the sequence spaces
\[
\ell_0^q(\beta) = \left \{
\begin{array}{llllll}
\ell^q(\beta) \text{, whenever } q<\infty \\
c_0(\beta)=\{\{a_n\}:\lim_{n\to\infty}a_nb_n=0\} \text{, if } q=+\infty
 \end{array}\right. .
\]
These spaces appear here because, to use the reiteration property with an approximation space $A(X,\{A_n\},S)$, it is necessary that $\bigcup_nA_n$ be dense in $A(X,\{A_n\},S)$ and, if $S=\ell^q(\beta)$ with $\|\beta\|_{\ell^q}=+\infty$, then the closure of $\bigcup_nA_n$ in  $A(X,\{A_n\},\ell^q(\beta))$ is  $A(X,\{A_n\},\ell_0^q(\beta))$.

\begin{theorem} \label{comp_appr_spa_dim_fin} Assume that $\|\beta\|_{\ell^q}=\infty$, where $\beta=\{b_n\}$ is a sequence of nonnegative real numbers and $b_0>0$, and let $(X,\{A_n\})$ be a linear approximation scheme with $\dim A_n<\infty$ for all $n$. The following assertions are equivalent:
\begin{itemize}
\item[$(i)$] $M$ is a relatively compact subset of  $A(X,\{A_n\},\ell^{q}_0(\beta))$.
\item[$(ii)$] There exists a sequence of nonnegative real numbers  $\gamma=\{a_n\}$ such that $a_0>0$, $\lim_{n\to\infty}a_n=\infty$, and  $M$ is a bounded subset of
 $A(X,\{A_n\},\ell^{q}_0(\{a_nb_n\}))$.
\end{itemize}
\end{theorem}

\begin{theorem}
Let us assume the hypotheses of Theorem \ref{comp_appr_spa_dim_fin}. The sequence $\{f_n\}\subseteq  A(X,\{A_n\},\ell^{q}_0(\beta))$ is convergent in the norm of $A(X,\{A_n\},\ell^{q}_0(\beta))$ if and only if it is convergent in the norm of $X$ and it forms a relatively compact subset of $A(X,\{A_n\},\ell^{q}_0(\beta))$.
\end{theorem}

\subsection{Characterization of compactness with arbitrary linear approximation sche\-mes and some applications}
So far, we have imposed over $A_n$ being boundedly compact or, even more, being a finite dimensional linear space. Obviously, these impositions were necessary for our proofs, but it is also true that they are strong assumptions. Is it possible, for example, to give some compactness criterium by using linear approximation schemes $(X,\{A_n\})$ if we allow $\dim A_n=\infty$?
Obviously, in those cases the characterization of compactness should be more complicated since being bounded in $A(X,\{A_n\},\ell^{q}_0(\beta))$ will not be a sufficient condition for a bounded subset of $X$ in order to be relatively compact. The reason is simple: the unit ball of $A_n$, which is not relatively compact since $A_n$ is infinite dimensional, is bounded in  $A(X,\{A_n\},\ell^{q}_0(\beta))$ for all $\beta$. Now, Almira and Luther \cite{almiraluther1} proved that, if $M$ is a bounded subset of  $A(X,\{A_n\},\ell^{q}_0(\beta))$, then compactness of $M$ as a subset of $X$ will follow from some extra assumptions.

\begin{theorem} \label{comp_dim_inf}
Let $(X,\{A_k\})$ be a linear approximation scheme and assume that there exist linear projections $P_k:X\to X$ with $P_k(X)=A_k$ for all $k\in\mathbb{N}$, and $\sup_{k\in\mathbb{N}}\|P_k\|=K<\infty$. Given  $M\subseteq X$ and $q\in [1,\infty ]$, the following are equivalent statements:
\begin{itemize}
\item[$(i)$] $M$ is relatively compact in $X$.
\item[$(ii)$] $P_k(M)$ is relatively compact in $X$ for $k=1,2\cdots$ and there exists $\beta=\{b_k\}\subseteq [0,\infty)$ such that $\|\beta\|_{\ell^q}=+\infty$, $b_0>0$ and $M$ is bounded in $A(X,\{A_k\},\ell^q(\beta))$.
\end{itemize}
\end{theorem}

\begin{proof}
The implication $(i)\Rightarrow (ii)$ is trivial. Indeed, it follows from $(i)\Rightarrow (ii)$ in Theorem \ref{timan} -which holds true for arbitrary approximation schemes $\{A_k\}$- that, if $M$ is relatively compact in $X$ then $\{E(M,A_k)\}\searrow 0$ and this is precisely what we need for the existence of the sequence $\beta$ with the desired properties. Furthermore, if $\{P_k(f_s)\}_{s=0}^\infty$ is an infinite sequence in $P_k(M)$ then $\{f_s\}$ is also an infinite sequence in $M$, so that it admits a convergent subsequence $\{f_{s_i}\}_{i=0}^\infty$. Obviously,  $\{P_k(f_{s_i})\}$ is also convergent since $\|P_k(f_{s_i})-P_k(f_{s_j})\|\leq \|P_k\|\|f_{s_i}-f_{s_j}\|$, which implies that $\{P_k(f_{s_i})\}$ is a Cauchy sequence.  This proves that $P_k(M)$ is relatively compact in $X$ for all $k$.

Let us prove $(ii)\Rightarrow (i)$. Let $\{f_m\}_{m=0}^\infty \subseteq M$ be an infinite sequence. We must show that $\{f_m\}$ contains a convergent subsequence. Let us define, for each $k\in\mathbb{N}$, $f_{k,m}=P_k(f_m)$. For each $x_k\in A_k$ we have that
\begin{eqnarray*}
\|f_m-P_k(f_m)\| &=& \|f_m-x_k-P_k(f_m-x_k)\|\leq \|I-P_k\|\|f_m-x_k\|\\
&\leq & (1+\|P_k\|)\|f_m-x_k\|.
\end{eqnarray*}
Thus, if we take the infimum between the elements $x_k\in A_k$, we get
\[
\|f_m-P_k(f_m)\|\leq (1+\|P_k\|)E(f_m,A_k)   \text{ for all } m,k\in\mathbb{N}.
\]
If we set $a_k=\displaystyle \frac{1+\|P_k\|}{\|\{b_i\}_{i=0}^k\|_{\ell^q}}$ and $c_k=\|\{b_i\}_{i=0}^k\|_{\ell^q}$, this inequality implies that
\begin{equation} \label{nueva}
\|f_m-P_k(f_m)\|\leq a_k(E(f_m,A_k)c_k+1) \text{ for all } m,k\in\mathbb{N}.
\end{equation}
Now, fixed $k\in\mathbb{N}$, every subsequence of $\{P_{k}(f_m)\}_{m=0}^\infty$ contains a convergent subsequence, since $P_k(M)$ is relatively compact in $X$, by hypothesis. In particular, the subsequence can be assumed to be of the form $\{P_k(f_m)\}_{m\in \mathbb{M}_0}$ ($\mathbb{M}_0$ an infinite subset of $\mathbb{N}$) and to satisfy the inequality
\begin{equation} \label{nuevados}
\|P_{k}(f_s)-P_{k}(f_t)\| \leq a_k[ (E(f_s,A_k)+E(f_t,A_k))c_k+1] +\varepsilon \,\,
\text{ for all } t,s\in\mathbb{M}_0\setminus [0,m_0(\varepsilon)),
\end{equation}
where $\varepsilon>0$ can be arbitrarily small and $m_0=m_0(\varepsilon)$ may depend on $\varepsilon$.

Using jointly the inequalities $\eqref{nueva}$ and $\eqref{nuevados}$, and the triangle inequality,
$$
\|f_t-f_s\|\leq \|f_t-P_k(f_t)\|+\|P_k(f_t)-P_k(f_s)\|+\|f_s-P_k(f_s)\|,
$$
we have that
\begin{equation} \label{nuevatres}
\|f_t-f_s\| \leq a_k[ 2(E(f_s,A_k)+E(f_t,A_k))c_k+3] +\varepsilon \\
\text{ for all } t,s\in\mathbb{M}_0\setminus [0,m_0(\varepsilon)).
\end{equation}
Obviously, $\|b\|_{\ell^q}=\infty$ and $\sup_{k\in\mathbb{N}}\|P_k\|=K<\infty$ imply that $\{a_k\}\in c_0$. Furthermore, the boudedness of $M$ in $A(X,\{A_k\},\ell^q(\beta))$ implies that $$(E(f_m,A_k)c_k)^q=E(f_m,A_k)^q\|\{b_i\}_{i=0}^k\|_{\ell^q}^q\leq \|f_m\|_{ A(X,\{A_k\},\ell^q(\beta))}^q\leq C^q$$ for all $m,k$ and a certain constant $C>0$.  Let us take $n_1<n_2<\cdots <n_i<\cdots$ a sequence of natural numbers such that $a_{n_i}(4C+3)\leq 2^{-i}$ for all $i=1,2,\cdots$, and a nested sequence of infinite sets $\mathbb{M}_i\subseteq \mathbb{N}$, $\mathbb{M}_{i+1}\subseteq \mathbb{M}_i$ for all $i$, such that
\begin{equation} \label{nuevacuatro}
\|f_t-f_s\| \leq a_{n_i}[ 2(E(f_s,A_{n_i})+E(f_t,A_{n_i}))c_{n_i}+3] +2^{-i} \leq 2^{-i}+2^{-i}=2^{-i+1}\\
\text{ for all } t,s\in\mathbb{M}_i.
\end{equation}
Then, if we choose $m_0<m_1<\cdots $ natural numbers such that $m_i\in\mathbb{M}_i$ for all $i$, the sequence $\{f_{m_i}\}$ satisfies $\|f_{m_i}-f_{m_j}\|\leq 2^{-i+1}$ for all $j\geq i$, so that it is a Cauchy sequence. This proves that $M$ is relatively compact.
\end{proof}

Obviously, if $\dim A_k<\infty$ for all $k$ and $M$ is a bounded subset of $X$, then $P_k(M)$ is relatively compact for all $k$. In this sense, Theorem \ref{comp_dim_inf} is clearly a generalization of Theorem \ref{timan}. On the other hand, if $X=H$ is a Hilbert space and $A_k$ is a closed subspace of $H$ for all $k$, then the orthogonal projections $P_k:H\to H$ ($P_k(H)=A_k$)  satisfy $\|P_k\|=1$ for all $k$, so that Theorem \ref{comp_dim_inf} can be useful in this context for arbitrary linear approximation schemes. In fact, in their paper \cite[Theorem 7.2]{almiraluther1}, the authors used this result to give a new proof of
Tjuriemskih's lethargy theorem \cite{tjuriemskih1}, \cite{tjuriemskih2} (see also  \cite{almiraluther2}, \cite{singerlibro}):

\begin{theorem}[Tjuriemskih] Let $(X,\{A_n\})$ be a nontrivial linear approximation scheme. Let $\{\varepsilon_n\}\searrow 0$ be a non-increasing sequence of positive numbers converging to zero, and let us assume that at least one of the following two conditions is fulfilled:
\begin{itemize}
\item[$(a)$] $\dim A_k<\infty$ for all $k\in\mathbb{N}$.
\item[$(b)$] $X$ is a Hilbert space.
\end{itemize}
Then there exists $f\in X$ such that $E(f,A_k)=\varepsilon_k$ for all $k\in\mathbb{N}$.
\end{theorem}

Furthermore, in \cite[Theorem 7.11]{almiraluther1}, Theorem \ref{comp_dim_inf} above were also used to prove a compactness criterium, which generalizes Kolmogorov's characterization of compactness in $L^p(\mathbb{R}^d)$ \cite{kolmogorov1931} (see also \cite[Theorem 5]{HOH}) and Simon's characterization of compactness for $L^p((0,T),X)$ \cite[Theorem 3.1]{simon}, for the spaces
$$L^p(\mathbb{R}^d,X)=\{f:\mathbb{R}^d\to X: f \text{ is measurable and} \|f\|_{L^p(\mathbb{R}^d,X)}=\left(\int_{\mathbb{R}^d}\|f(x)\|_X^pdx\right)^p<\infty\},$$
where $X$ is  (any)  Banach space.

\begin{theorem} A bounded set $M\subseteq L^p(\mathbb{R}^d,X)$ is relatively compact in  $L^p(\mathbb{R}^d,X)$ if and only if the following three conditions are satisfied:
\begin{itemize}
\item[$(i)$] $\lim_{k\to\infty}\int_{\|x\|\geq k}\|f(x)\|_X^pdx=0$ uniformly in $f\in M$.
\item[$(ii)$] The set $\{\int_{[a,b]}f(x)dx:f\in M\}\subseteq X$ is relatively compact for all $a,b\in\mathbb{R}^d$ with $a<b$. (This means that $a=(a_1,\cdots,a_d)$, $b=(b_1,\cdots,b_d)$ satisfy $a_i< b_i$ for all $i$, and $[a,b]=[a_1,b_1]\times \cdots\times [a_d,b_d]$).
\item[$(iii)$] $\lim_{\|h\|\to 0}\|f(\cdot+h)-f(\cdot)\|_p=0$ uniformly in $f\in M$.
\end{itemize}

\end{theorem}

A careful inspection of the proof of Theorem \ref{comp_dim_inf} reveals that the important steps for its arguments are the inequalities $\eqref{nueva}$ and $\eqref{nuevados}$. This directly leads to introduce the following technical concept, and the reformulation of the result below it.
\begin{defi}[$(\beta,q)$-condition] Let $q\in [1,\infty]$ and assume that $\beta=\{b_k\}_{k=0}^\infty$ is a sequence of positive numbers such that $b_0>0$ and $\|\beta\|_{\ell^q}=\infty$. Let $(X,\{A_n\})$ be a linear approximation scheme and let us assume that $M\subset X$. We say that $M$ satisfies the $(\beta,q)$-condition with respect to  $(X,\{A_n\})$, if for every sequence $\{f_m\}\subseteq M$ there exist  $\{a_k\}\subseteq [0,\infty)$ and sequences $\{f_{k,m}\}_{m=1}^\infty\subseteq X$, $k\in\mathbb{N}$, such that
\begin{itemize}
\item[$(i)$] $\{a_n\}\in c_0$
\item[$(ii)$] $\|f_m-f_{k,m}\|_X\leq a_k[ E(f_m,A_k)\|\{b_i\}_{i=0}^k\|_{\ell^q}+1]$ for all $m,k\in\mathbb{N}$.
\item[$(iii)$] For all $k\in\mathbb{N}$, every subsequence of $\{f_{k,m}\}_{m=0}^\infty$ contains a subsequence $\{f_{k,m}\}_{m\in\mathbb{M}_0}$ ($\mathbb{M}_0$ been an infinite subset of $\mathbb{N}$) such that
\[
\|f_{k,s}-f_{k,t}\|_X\leq a_k[ (E(f_s,A_k)+E(f_t,A_k))\|\{b_i\}_{i=0}^k\|_{\ell^q}+1] \text{ for all } t,s\in\mathbb{M}_0\setminus [0,m_0(\varepsilon)),
\]
where $\varepsilon_0>0$ is arbitrarily small and $m_0=m_0(\varepsilon)$ may depend on $\varepsilon$.
\end{itemize}
\end{defi}

\begin{theorem} \label{comp_dim_inf_gen}
Let $(X,\{A_k\})$ be a linear approximation scheme and let $q\in [1,\infty ]$ be fixed.  The following are equivalent statements:
\begin{itemize}
\item[$(i)$] $M$ is relatively compact in $X$.
\item[$(ii)$] There exists $\beta=\{b_k\}\subseteq [0,\infty)$ such that $\|\beta\|_{\ell^q}=+\infty$, $b_0>0$ and $M$ is a bounded subset of  $A(X,\{A_k\},\ell^q(\beta))$ which satisfies the $(\beta,q)$-condition with respect to  $(X,\{A_n\})$.
\end{itemize}
\end{theorem}

\section{Generalized approximation schemes and $Q$-compactness}
\subsection{Preliminaries. A few examples}
%\textcolor{red}{\begin{defi}[Approximation Scheme]
%Let $X$ be a Banach space over the field $K$ of real or complex numbers and $N$ be the set of all non-negative integers.  For each $n\in N$, let $Q_n=Q_n(X)$ be %a family of subsets of $X$ satisfying the following conditions:
%\begin{enumerate}[(1)]
%\item $\{0\}=Q_0\subset Q_1\subset \cdots \subset Q_n \subset \dots;$
%\item $\lambda Q_n \subset Q_n$ for ever $n \in N$ and $\lambda \in K;$
%\item $Q_n+Q_m \in Q_{n+m}$ for every $n,m \in N$.
%\end{enumerate}
%Then $Q(X)=(Q_n(X))_{n \in N}$ is called an \emph{approximation scheme} on $X$.  We shall simply use $Q_n$ to denote $Q_n(X)$ if the context is clear.
%\end{defi}

\begin{defi}[Generalized Approximation Scheme]
Let $X$ be a Banach space.  For each $n\in \mathbb{N}$, let $Q_n=Q_n(X)$ be a family of subsets of $X$ satisfying the following conditions:
\begin{itemize}
\item[$(GA1)$] $\{0\}=Q_0\subset Q_1\subset \cdots \subset Q_n \subset \dots$.
\item[$(GA2)$]  $\lambda Q_n \subset Q_n$ for all $n \in N$ and all scalars $\lambda$.
\item[$(GA3)$]  $Q_n+Q_m \subseteq Q_{n+m}$ for every $n,m \in N$.
\end{itemize}
Then $Q(X)=(Q_n(X))_{n \in N}$ is called a \emph{generalized approximation scheme} on $X$.  We shall simply use $Q_n$ to denote $Q_n(X)$ if the context is clear.
\end{defi}
Obviously, there are several important differences between this concept and Definition \ref{ap_sch}  and, in fact, no one of these concepts includes the other one. We use here the term ``generalized'' because the elements of $Q_n$ may be subsets of $X$ (and not just elements of $X$, as it was the case in Definition \ref{ap_sch}).

%\subsection{Examples}

Let us now consider a few important examples of generalized approximation schemes:
\begin{enumerate}[1)]
\item The classical approximation schemes introduced in Pietsch in his seminal paper \cite{Pie}.
\item  $Q_n=$ the set of all at-most-$n$-dimensional subspaces of any given Banach space $X$.
\item Let $E$ be a Banach space and $X=L(E)$; let $Q_n=N_n(E)$, where $N_n(E)=$ the set of all $n$-nuclear maps on $E$. \cite {PieID}
\item Let $a^k=(a_n)^{1+\frac{1}{k}},$ where $(a_n)$ is a nuclear exponent sequence. Then {$Q_n$}  on $X=L(E)$ can be defined as the set of all $\Lambda_\infty (a^k)$-nuclear maps on $E$.\cite{Dubinsky_Ram}
\end{enumerate}
We are now able to introduce $Q$-compact sets and operators:
\begin{defi}[Generalized Kolmogorov Number]
Let $U_X$ be the closed unit ball of $X$,  $Q(X)=(Q_n(X))_{n \in N}$ be a \emph{generalized approximation scheme} on $X$,  and $D$ be a bounded subset of $X$.  Then the $n^{\text{th}}$ \emph{generalized Kolmogorov number} $\delta_n(D;Q)$ of $D$ with respect to $U_X$ is defined by
\begin{equation}
\label{GenKolmogorovNumber}
\delta_n(D;Q)=\inf\{r>0:D \subset rU_X+A \text{ for some }A \in Q_n(X)\}.
\end{equation}
Assume that $Y$ is a Banach space and $T \in L(Y,X)$. The $n^{\text{th}}$ Kolmogorov number $\delta_n(T;Q)$ of $T$ is defined as $\delta_n(T(U_Y);Q)$.
\end{defi}
It follows that $\delta_n(T;Q)$ forms a non-increasing sequence on non-negative numbers:
\begin{equation}
\|T\|=\delta_0(T;Q)\geq \delta_1(T;Q)\geq \cdots \geq \delta_n(T;Q)\geq 0.
\end{equation}
\begin{defi}[$Q$-compact set]
Let $D$ be a bounded subset of $X$. We say that $D$ is $Q$-\emph{compact} if $\displaystyle\lim_n \delta_n(D;Q)=0$.
\end{defi}

\begin{defi}[$Q$-Compact Operator] We say that $T\in L(Y,X)$ is a $Q$-\emph{compact operator} if $\displaystyle\lim_n \delta_n(T;Q)=0$, i.e., $T(U_Y)$ is a $Q$-compact set.
\end{defi}

\begin{rem} If $Q=\{A_n\}_{n=0}^\infty$ is a classical approximation scheme, then $A\in A_n$ means that $A=A_n$, so that, for any set $D\subseteq X$, $\delta_n(D;\{A_n\})=E(D,A_n)$, since
\begin{eqnarray*}
\delta_n(D;\{A_n\}) &=& \inf\{r:D\subseteq rU_X+A_n\}\\
 &=& \inf\{r:E(x,A_n)\leq r \text{ for all }x\in D\}\\
 &=& \sup_{x\in D}E(x,A_n)=E(D,A_n).
\end{eqnarray*}
Hence, in this case $D\subseteq X$ is $\{A_n\}$-compact if and only if $\{E(D,A_n)\}\searrow 0$ and Theorem $\ref{timan}$ states that, if $A_n$ is boundedly compact for all $n$, then $D\subseteq X$ is relatively compact in $X$ if and only if it is bounded in $X$ and $\{A_n\}$-compact. Indeed, all theorems in Section $2$ of the paper are also results about $Q$-compact sets or operators. For example, Corollary \ref{A_nCompatness}  characterizes $\{A_n\}$-compactness of subsets of $X$ whenever $\{A_n\}$ is a boundedly compact approximation scheme on $X$.
\end{rem}

\begin{prop}
Let $Q=\{Q_n(X)\}$ be a generalized approximation scheme on $X$ and assume that all elements $A\in Q_n$ are cones (i.e., $\lambda A\subseteq A$ for all scalar $\lambda$), for $n=1,2,\cdots$. If $X$ is separable and $\{\delta_n(\{x\};Q)\}\searrow 0$ for all $x\in X$, then all relatively compact subsets of $X$ are $Q$-compact sets.
\end{prop}

\begin{proof}
Let $\{x_n\}_{n=0}^\infty$ be a countable dense  subset of $X$. For each $n,m\in \mathbb{N}$, we take $A_{n,m}\in Q_m$ and $a_{n,m}\in A_{n,m}$ such that
\[
\|x_n-a_{n,m}\|\leq 2E(x_n,A_{n,m})\leq 3\delta_{m}(\{x_n\},Q).
\]
Then $\{a_{n,m}\}_{n,m\in\mathbb{N}}$ is dense in $X$ since $\lim_{m\to\infty}\delta_{m}(\{x_n\},Q)\to 0$ for all $n\in\mathbb{N}$ and $\{x_n\}$ is dense in $X$. It follows that
$\bigcup_{N=0}^\infty B_N$ is dense in $X$, where $B_N=\mathbf{span} \{a_{n,m}\}_{n,m=1}^N$ is a linear subspace of $X$ for all $N$. This obviously implies that, taking $B_0=\{0\}$, the family  $\{B_n\}_{n=0}^\infty$ is a linear approximation scheme of $X$.
On the other hand, it follows from $(GA3)$ that, for each $N$, there exists $K(N)\geq N$ and $\widetilde{A}_{K(N)}\in Q_{K(N)}$ such that
$A_{1,1}+A_{1,2}+\cdots+A_{N,N}\subseteq \widetilde{A}_{K(N)}$. Furthermore, this implies that $B_N\subseteq \widetilde{A}_{K(N)}$ since the sets $A_{n,m}$ are cones. Hence
\begin{equation} \label{CIQ}
\delta_{K(N)}(M;Q)\leq E(M,\widetilde{A}_{K(N)})\leq E(M,B_N), \text{ for all } M\subseteq X\text{ and all } N\in\mathbb{N}.
\end{equation}
We claim that if $M$ is relatively compact in $X$, then $\{E(M,B_n)\}\searrow 0$.  To prove this result, let us assume that the contrary is true. Then there exist $\{y_n\}_{n=1}^\infty \subset M$ and  $c>0$ such that $E(y_n,B_n)>c$ for all $n$. The relative compactness of $M$ implies that there exists a subsequence $\{y_{n_k}\}_{k=1}^\infty$ and $y\in X$ such that $\lim_{k\to\infty}\|y_{n_k}-y\|=0$. Hence
\[
E(y_{n_k},B_{n_k})\leq E(y_{n_k}-y,B_{n_k})+E(y,B_{n_k})\leq \|y_{n_k}-y\|+E(y,B_{n_k})\to 0 (\text{ for } k\to\infty),
\]
which contradicts $c<E(y_{n_k},B_{n_k})$, $k=1,2,\cdots$. It follows that  $\{E(M,B_n)\}\searrow 0$ and the inequalities $\eqref{CIQ}$ imply that $M$ is $Q$-compact.
\end{proof}

\subsection{$Q$-Compactness Does Not Imply Compactness}
In this section we show that in $L_p[0,1],2\leq p\leq \infty$, with a suitably defined approximation scheme, we can find a $Q$-compact map which is not compact.
\par
Let $[r_n]$ be the space spanned by the Rademacher functions.  It can be seen from the Khinchin Inequality that
\begin{equation}
\ell^2 \approx [r_n]\subset L_p[0,1] \text{ for all }1\leq p \leq \infty.
\end{equation}
We define an approximation scheme $A_n$ on $L_p[0,1]$ as follows:
\begin{equation}
A_n=\{f\in L_p[0,1]: f \in L_{p+\frac{1}{n}}\} \text{ or simply } A_n=L_{p+\frac{1}{n}}.
\end{equation}
$L_{p+\frac{1}{n}}\subset L_{p+\frac{1}{n+1}}$ gives us $A_n\subset A_{n+1}$. for $n=1,2,\dots,$ and it is easily seen that $A_n+A_m \subset A_{n+m}$ for $n,m=1,2,\dots,$ and that $\lambda A_n \subset A_n$.  Thus $\{A_n\}$ is an approximation scheme in the sense of Pietsch.
\par
Next we observe the existence of a projection
\begin{equation}
\underline{P}:L_p[0,1]\to R_p \text{ for }p\geq 2,
\end{equation}
where $R_p$ denotes the closure of the span of $\{r_n(t)\}$ in $L_p[0,1]$.  We know that for $p\geq 2$, $L_p[0,1]\subset L_2[0,1]$.  Now $R_2$ is a closed subspace of $L_2[0,1]$ and $\underline{P}_2:L_2[0,1]\to R_2$ is an orthogonal projection onto $R_2$.  Then $\underline{P}=j\circ \underline{P}_2\circ i$, where $i,j$ are isomorphisms shown in the Figure %\ref{fig:ij}, is clearly a projection.
\begin{figure}[H]
\label{fig:ij}
\centering
\begin{tikzpicture}[
    path/.style={thick},
    every node/.style={color=black}
    ]
       % Lines
    \draw[path,->] (-1,1) -- (1,1) node[right] {$L_2$}node[midway,above]{$i$};
    \draw[path,->] (1.3,.7)--(1.3,-.7) node[midway,right]{$\underline{P}_2$};
    \draw[path,<-] (-1,-1) -- (1,-1)node[right] {$R_2$}node[midway,below]{$j$};
    \draw[path,->] (-1.3,.7)--(-1.3,-.7)node[midway,left]{$\underline{P}$}node[below]{$R_p$};
    \node at (-1.3,1){$L_p$};
\end{tikzpicture}
\end{figure}
\begin{prop}
For $p\geq 2$ the projection $\underline{P}:L_p[0,1]\to R_p$ is $Q$-compact but not compact.
\end{prop}
\begin{proof}
Let $U_{R_p},U_{L_p}$ denote the closed unit balls of $R_p$ and $L_p$ respectively.  It is easily seen that $\underline{P}(U_{L_p})\subset \|\underline{P}\|U_{R_p}$.  But $U_{R_p} \subset CU_{R_{P+\frac{1}{n}}}$ where $C$ is a constant follows from the Khinchin inequality.  Therefore, $\underline{P}(U_{L_p})\subset L_{p+\frac{1}{n}}$, which gives $\delta_n(P,Q)\to 0$.  To see that $\underline{P}$ is not a compact operator, observe that dim$R_p=\infty$ and $I-\underline{P}$ is projection with kernel $R_p$, so $I-\underline{P}$ is not a Fredholm operator.  Therefore $\underline{P}$ is not a Riesz operator, but every compact operator is a Riesz operator. So $\underline{P}$ cannot be a compact operator.
\end{proof}

\begin{rem}
Another example which proves that $Q$-compactness does not imply compactness: Take $X=H$ a Hilbert space, $Y\subset H$ an infinite dimensional closed subspace such that $\dim H/Y=\infty$, $D=U_Y$ and $T=P_Y:H\to H$ the orthogonal projection of $H$ onto  $Y$. Take $\{A_n\}$ any nontrivial linear approximation scheme on $H$ such that $A_0=\{0\}$ and $A_1=Y$.  Then $D=T(U_X)=U_Y$ is not relatively compact in $X$ (so that $T$ is not a compact operator) and $\delta_n(T,\{A_n\})=\delta_n(D;\{A_n\})=E(D,A_n)=0$ for all $n\geq 1$, so that $T$ and $D$ are $\{A_n\}$-compact.
\end{rem}

\subsection{Properties of $Q$-Compact Maps}
%For a given approximation scheme $Q_n$ on $X$ we shall define a continuous linear map $T \in L(X)$ to be $Q$-compact if $T(U_X)$ is $Q$-compact in $X$ or %equivalently if $\displaystyle \lim_n \delta_n(T(U_X);Q)=\lim_n\delta_n(T;Q)=0$.
%\par
Let $\mathcal{A}$ be the ideal defined as
\begin{equation}
\mathcal{A}=\{T \in L(X): \delta_n(T;Q) \to 0 \text{ as } n\to\infty\},
\end{equation}
and let $\mathcal{A}^s$ denote the surjective hull of $\mathcal{A}$, which is defined by
\begin{equation}
\mathcal{A}^s=\{T \in L(X):\delta_n(TQ_{E^1};Q) \to 0 \text{ as } n\to\infty\}.
\end{equation}
where $Q_{E^1}$ is a surjection of $\ell_I^1$ with $Q_{E^1}(U_{\ell_I^1})=U_X$.
\begin{prop}
\mbox{}
\begin{enumerate}[i)]
\item $Q$-compact maps have separable range;
\item the uniform limit of $Q$-compact maps is $Q$-compact;
\item an ideal of $Q$-compact maps is equal to its surjective hull, i.e. $\mathcal{A}=\mathcal{A}^s$.
\end{enumerate}
\end{prop}
\begin{proof}
i) Follows from the definition.  For ii) we first observe that $\delta_0(T;Q\leq \|T\|$.  Now suppose $(T_n)$ is a sequence of $Q$-compact maps, and let $T=\displaystyle \lim_{n}T_n$.  Then
\begin{gather}
\delta_n(T;Q)=\delta_n(T-T_n+T_n;Q)\leq \delta_0(T-T_n;Q)+\delta_n(T_n;Q)\notag\\
\leq \|T-T_n\|+\delta_n(T_n;Q)
\end{gather}
which gives that $T$ is $Q$-compact too.
\par
For iii), $\mathcal{A} \subset \mathcal{A}^c$ follows from the fact that
\begin{equation}
\delta_n(TQ_{E^1};Q)\leq \delta_n(T;Q)\|Q_{E^1}\|=\delta_n(T;Q);
\end{equation}
on the other hand
\begin{equation}
\delta_n(TQ_{E^1};Q)\leq \delta_n(TQ_{E^1}(U_{\ell_I^1});Q)=\delta_n(T;Q);
\end{equation}
gives the equality readily.
\end{proof}
%%%%%%%%%%%%%%%%%%%%%%%%%%%%
\begin{rem}
 Let $T$ be a linear mapping from a Banach space $X$ into a Banach space $Y$. According a classical theorem of Schauder (\cite {Dun-Sch}, p.$485$) an operator $T\in L(X,Y)$  is compact if and only if its adjoint $T^* \in L(Y^*, X^*)$ is compact. Using Schauder theorem Terzio\~{g}lu \cite{Terzioglu} gave  a representation theorem for compact maps. He proves that $T\in L(X,Y)$ is compact if and only if there is a sequence $(u_n)$ of continuous linear functionals on $X$ with $\displaystyle\lim_{n} || u_n|| =0$ such that the inequality $$ ||Tx|| \leq \sup_n |<u_n, x>|$$ holds for every $x\in X$. In general Schauder type of theorem need not be true for $Q$-compact maps. However a result analogous to Terzio\~{g}lu's can be proved for $Q$-compact maps if one assumes both $T$ and $T^*$ are $Q$-compact. For details see \cite{Ak0}.
\end{rem}

%%%%%%%%%%%%%%%%%%%%%
\subsection{Q-Compact Sets}
%\begin{defi}[$Q$-Compact Set]
%A bounded subset $D$ of $X$ is said to be a $Q$\emph{-compact set} if $\displaystyle \lim_n\delta_n(D;Q)=0$.
%\end{defi}
%\begin{defi}{$Q$-Compact Set}
%Let $X$ be a Banach space.  A bounded subset $D$ of $X$ is said to be $Q$\emph{-compact} if $\delta_n(D;Q)\to 0 \ (n\to\infty)$.
%\end{defi}
We assume each $A_n \in Q_n(n \in N)$ is separable.  It is immediate from the definitions that $Q$-compact sets are separable and $Q$-compact maps have separable range.
\begin{defi}[Order-$c_0$-sequence]
A double sequence $\{x_{n,k}\}_{n,k\in \mathbb{N}} \subset X$ is said to be an $\emph{order-$c_0$-sequence}$ if the following hold:
\begin{enumerate}[(1)]
\item for every $n \in \mathbb{N}$ there exists an $A_n \in Q_n$ such that $\{x_{n,k}\}_{k=0}^{\infty} \subset A_n$;
\item $\|x_{n,k}\|\to 0$ as $n\to \infty$ uniformly in $k$.
\end{enumerate}
\end{defi}
\begin{theorem}
\label{thm:qcompact}
Suppose $(X,Q_n)$ is a generalized approximation scheme with sets $A_n \in Q_n$ assumed to be solid (i.e, $tA_n \subset A_n$ for all $t \in [0,1]$).  Then a bounded subset $D$ of $X$ is $Q$-compact if and only if there exists an order-$c_0$-sequence $\{x_{n,k}\}_{k=0}^{\infty} \subset X$ such that
\begin{equation}
D\subset \left\{\sum_{n=1}^\infty \lambda_nx_{n,k(n)}: \text{ } \{k(n)\}_{n=0}^\infty \subseteq \mathbb{N}\text{ and } \sum_{n=1}^\infty|\lambda_n|\leq 1\right\}.
\end{equation}
\end{theorem}
\begin{proof}
Let $D$ be $Q$-compact.  Then $\delta_n(2D,Q)\to 0$ and so there exists $n_1$ such that
\begin{equation}
2D \subset \frac{1}{4}U_X+A_{n_1}.
\end{equation}
Since $A_{n_1}$ is separable let $\{x_{1,k}\}_{k=0}^\infty$ be a countable dense subset of $A_{n_1}$; then it is easy to see that $B_1=(2D+\frac{1}{2}U_X)\cap\{x_{1,k}\}_{k=0}^{\infty} \neq \emptyset $ (and is an infinite countable set) and $2D\subset B_1+\frac{1}{2}U_X$.
\par
Let $D_1=(2D-B_1)\cap \frac{1}{2}U_X$, where $2D-B_1$ is the ordinary vector difference.  Then $D_1$ is a bounded set (since it is a subset of $\frac{1}{2}U$) and given $\epsilon>0$ we get, by the $Q$-compactness of $2D$, that $2D-B_1\subset \epsilon U_X+A_m+\tilde{A}_{n_1}\subset \tilde{\tilde{A}}_{m+n_1}+\epsilon U_X$ for suitable $m$ and suitable $\tilde{A}_{n_1}\in Q_{n_1}$, $\tilde{\tilde{A}}_{m+n_1} \in Q_{m+n_1}$; this is true because $B_1\subset \tilde{A}_{n_1}$ and $\lambda\tilde{A}_{n_1}\in Q_{n_1}$ for each $\lambda$.  This shows that $D_1$ is $Q$-compact and, as before, there exists $A_{n_2}\in Q_{n_2}$ such that $2D_1 \subset \frac{1}{8}U_X+A_{n_2}$. Let $\{x_{2,k}\}_{k=0}^{\infty}$ be a dense subset of $A_{n_2}$.  Then
\begin{gather}
B_2=(2D_1+\frac{1}{4}U_X)\cap \{x_{2,k}\}_{k=0}^{\infty} \text{ is infinite countable};\\
2D_1\subset B_2+\frac{1}{4}U_X;\\
D_2=(2D_1-B_2)\cap \frac{1}{4}U_X \text{ is }Q\text{-compact}.
\end{gather}
Continuing this process we define
\begin{equation}
B_m=\left(2D_{m-1}+\frac{1}{2^m}U_X\right)\cap\{x_{m,k}\}_{k=0}^{\infty},\ \{x_{m,k}\}_{k=0}^{\infty} \text{ dense in } A_{n_{m}};
\end{equation}
then $2D_{m-1} \subset B_m+\frac{1}{2^m}U_X$ and we define
\begin{equation}
D_m=(2D_{m-1}-B_m)\cap \frac{1}{2^m}U_X.
\end{equation}
Our construction gives for each $d \in D$, successively chosen $b_i \in B_i, i=1,2,\dots,k$ such that
\begin{equation}
d-\left(\frac{1}{2}b_1+\frac{1}{2^2}b_2+\dots+\frac{1}{2^k}b_k\right)\in 2^{-k}D_k,
\end{equation}
and since $D_k \subset 2^{-k}U_X$, it follows that
\begin{equation}
d=\sum_{n=1}^\infty \frac{1}{2^n}b_n.
\end{equation}
Since each $b_n=x_{n,k(n)}$ for a suitable $k(n)$ and since $$b_n \in B_n \subset 2D_{n-1}+\frac{1}{2^n}U_X \subset 2 \cdot \frac{1}{2^{n-1}}U_X+\frac{1}{2^n}U_X\subset \frac{3}{2^{n-2}}U_X,$$ it follows that $\|b_n\|\to0.$
\par
In the reverse direction, suppose we have that for each $n$ an $A_n \in Q_n$ and $\{x_{n,k}\}_{k=0}^{\infty} \subset A_n$ with $\|x_{n,k}\|\to 0$ as $n\to \infty$ uniformly in $k$ and
\begin{equation}
D\subset \left\{\sum_n\lambda_nx_{n,k(n)}:\sum_{n=0}^\infty |\lambda_n|\leq 1\text{ and } \{k(n)\}_{n=0}^{\infty}\subseteq \mathbb{N} \right\}:=C.
\end{equation}
Since for each $c \in C$ we can write
\begin{equation}
c=\sum_{n=1}^m\lambda_nx_{n,k(n)}+\sum_{n=m+1}^\infty \lambda_nx_{n,k(n)}=u+v,
\end{equation}
where $u \in \lambda_1A_1+\dots +\lambda_m A_m$, our assumption on $Q_n$ and solidness of the $A_n$'s give that $u \in \tilde{A}_{m^2}$. Furthermore,  given
$\epsilon>0$ we may choose $m$ such that $\|x_{n,k}\|<\epsilon$ for each $k>m$.  Thus $C \subset \epsilon U+\tilde{A}_{m^2}$ and so $\delta_n(C,Q)\to 0$ as $n\to\infty$, and therefore, also $\delta_n(D,Q)\to 0$.
\end{proof}
\begin{rem}
Theorem \ref{thm:qcompact} can be considered as an analogue of the Dieudonne-Schwartz lemma on compact sets in terms of standard Kolmogorov diameter.  If one chooses $Q_n$ to be the at-most-$n$-dimensional subspaces of $X$ one can show that $Q$-compactness of a bounded subset $D$ coincides with the usual definition of compactness of $D$
\end{rem}
\begin{rem}
The first author and M.Nakamura have proven a similar theorem for $p$-normed spaces, $0\leq p \leq 1$.
\end{rem}
Next we give a characterization of $Q$-compact subsets of $X$ via $Q$-compact maps into $X$.
\begin{theorem}
\label{qsub}
Assume $(X,Q_n)$ is a generalized approximation scheme on the Banach space $X$ with each $A_n \in Q_n$ being a vector subspace of $X$. Then, a bounded subset $D$ of $X$ is $Q$-compact if and only if $D \subset T(U_E)$ for a suitable Banach space $E$ and a $Q$-compact map $T$ on $E$ into $X$.
\end{theorem}
\begin{proof}
We need only prove the ``only if'' part.  Let $D$ be $Q$-compact and let $C$ denote the closed absolute convex hull of $D$.  Then that $C$ is $Q$-compact is easily seen as follows: each $c \in C$ is of the form $\displaystyle c=\sum_{i=1}^m\lambda_id_i$ with $\displaystyle \sum_{i=1}^m|\lambda_i|\leq 1$ and $d_i \in D$ for each $i$; give $\epsilon>0$, there exists $N$ such that for all $n\geq N,\delta_n(D,Q)<\epsilon$ and equivalently $D \subset \epsilon U_X +A_n$ and obviously then $C \subset \epsilon U_X+A_n$.
\par
Let $X_C$ denote the linear subspace of $X$ spanned by the elements of $C$ endowed with the norm given by the gauge ( = Minkowski functional) $\mu$ of $C$.  Then $(X_C,\mu_C)$ is a Banach space (see, e.g., \cite{rolewicz}, \cite{rudin}).  Let $E=(X_C,\mu_C)$.  If $T$ is the canonical injection of $X_C$ into $X$, then $T(U_E)=C \supset D$ and $T$ is $Q$-compact.
\end{proof}
\begin{rem}
Using order $c_0$-sequences and associated sets $S_m= \{\displaystyle \sum_{n=1}^{m} \lambda_nx_n,k(n) :\,\,\, \displaystyle \sum_{n=1}^{m} |\lambda_n| \leq 1\}$,one can define the ball measure of non-$Q$-compactness $\gamma(D)$ of a bounded set $D$ in a Banach space $X$ as $ \gamma (D,Q)= \inf\{r>0:  D \subset \bigcup_{x\in S_n} B(x,r)\}$. It is shown in \cite{Ak0} that $$\gamma (D,Q)= \displaystyle \lim_{n} \delta(D, Q).$$ Furthermore, if we denote by $Q_c$ the ideal of $Q$-compact maps, then the ideal variation $\gamma_{Q_c} (D)= \inf\{r>0:\,\,\exists E\,\,\mbox{and}\,\,\, T\in Q_c(E,X)\, \,\mbox{such that}\,\, D \subset T(U_E)+rU_X\} = \gamma(D)$.
\end{rem}

\bigskip

\footnotesize{A. G. Aksoy

Department of Mathematics. Claremont McKenna College.

Claremont, CA, 91711, USA.

email:  aaksoy@cmc.edu}

\bigskip

\footnotesize{J. M. Almira

Departamento de Matem\'{a}ticas. Universidad de Ja\'{e}n.

E.P.S. Linares,  C/Alfonso X el Sabio, 28

23700 Linares (Ja\'{e}n) Spain

email: jmalmira@ujaen.es}

\bigskip

\end{document}